\documentclass[11pt]{amsart}

\usepackage{latexsym}
\usepackage{amssymb}
\usepackage{amsmath}
\usepackage{color}

\newtheorem{theorem}{Theorem}[section]

\newtheorem{proposition}[theorem]{Proposition}
\newtheorem{corollary}[theorem]{Corollary}
\newtheorem{definition}[theorem]{Definition}

\newtheorem{remark}[theorem]{Remark}

\newcommand{\cl}[1]{\mathcal{#1}}
\newcommand{\bb}[1]{\mathbb{#1}}

\begin{document}

\title{Distances for Operator-valued information channels}

\author[G. Baziotis]{Georgios Baziotis}
\address{
School of Mathematical Sciences, University of Delaware, 501 Ewing Hall,
Newark, DE 19716, USA}
\email{baziotis@udel.edu}

\date{\today}

\maketitle
\begin{abstract}
    We introduce three metrics on the set of quantum probability measures over a compact Hausdorff space and characterize them in terms of the completely bounded norm of the corresponding unital completely positive maps. We extend the existing topological structures between scalar-valued information channels to operator-valued ones and associate them with topologies on the set of unital completely positive maps between a commutative C*-algebra and the C*-algebra of bounded weakly measurable operator-valued functions over a compact Hausdorff space. Given a measure on the input alphabet space, we introduce the notion of an almost everywhere defined operator-valued information channel and provide a characterization result.
\end{abstract}
\section{Introduction}
Topological structures between classical information channels is a topic studied by several authors both in the setting of abstract and discrete input/output alphabet spaces \cite{jacobs,kieffer,neu-sh}. Their quantum counterparts initiated in \cite{btt} under the term operator-valued information channels are generalizations of finite families of POVM's which are of increasing importance in quantum information theory and specifically in non-local game theory as they are used to define the classes of quantum and quantum commuting no-signalling correlations. Motivated by the measure theoretic approach of non-local games studied in \cite{btt}, the investigation of the different modes of convergence between POVM's defined on spaces of arbitrary cardinality became of increasing importance in order to answer questions regarding the closure of various subclasses of measurable no-signalling correlations. However, we believe that this study is of independent interest since it generalizes the existing work in \cite{jacobs,neu-sh} to the non-commutative realm. 

\par We introduce an analogue of the total variation distance between two $\cl{B}(H)$-valued quantum probability measures and show that it coincides with the completely bounded norm of the difference of the corresponding unital completely positive maps. In the sequel, we use this metric characterization to provide a continuity theorem for Naimark's Dilation Theorem similar to the main result in \cite{werner}. Furthermore, we define a weak topology on the space of quantum probability measures and associate it with Arveson's BW topology \cite{arveson-acta,Pa}. 

\par  We study the modes of convergence between operator-valued information channels between two second countable compact Hausdorff spaces $X$ and $A$, viewed as the respective abstract input and output alphabets of the channel. We introduce a new metric called the $\tilde{\rho}$-metric on the space of all operator-valued information channels and show that the $\tilde{\rho}$-distance between two operator-valued information channels is equal to the operator norm of the difference of the corresponding unital completely positive maps from the space $C(A)$ of all complex-valued continuous functions on $A$ to the space $\cl{L}^\infty(X,\cl{B}(H))$ of all bounded weakly measurable functions from $X$ to $\cl{B}(H)$.
We further introduce a weaker mode of convergence and prove that it is sequentially compact.

\par Finally, given a probability measure $\mu$ on the abstract input alphabet space $X$, we introduce the notion of an operator-valued information channel mod $\mu$ as a Borel measurable family of quantum probability measures defined up to a $\mu$-null set and characterize them as unital completely positive maps from the space $C(A)$ to the space $L^\infty_\sigma(X,\cl{B}(H))$ of all essentially bounded weakly measurable functions from $X$ to $\cl{B}(H)$. 
\par 
For the remainder of this section we set the basic notation. Given a compact Hausdorff space $A$, we denote by $C(A)$ the unital C$^*$-algebra of all continuous functions defined on $A$. We write $\frak{A}_A$ for the Borel $\sigma$-algebra on $A$ and $M(A)$ for the space of all complex Borel measures on $A$. Note the isometric identification $C(A)^*=M(A)$. Let $H$ be a separable Hilbert space, $\cl{B}(H)$ will denote the space of all bounded operators on $H$ and $\cl{T}(H)$ the space of all trace class operators on $H$. For any $\xi,\eta\in H$ we will denote by $\xi\otimes \eta\in\cl{T}(H)$ the operator on $H$ given by $(\xi\otimes \eta)(\zeta)=\langle \zeta,\eta\rangle\xi$.
\par Let $\phi:\cl{A}\rightarrow\cl{B}$ be a linear map between two unital C$^*$-algebras $\cl{A}$ and $\cl{B}$, we say that $\phi$ is unital if $\phi(1_{\cl{A}})=1_{\cl{B}}$. For each $n\in\bb{N}$, set $\phi^{(n)}:M_n(\cl{A})\rightarrow M_n(\cl{B})$ to be the map defined as $\phi^{(n)}([a_{ij}]_{i,j})=[\phi(a_{i,j})]_{i,j}$ for $[a_{ij}]_{i,j}\in M_n(\cl{A})$. We say that $\phi$ is $n$-positive if $\phi^{(n)}(A)\in M_n(\cl{B})^+$ whenever $A\in M_n(\cl{A})^+$ and completely positive if $\phi$ is $n$-positive for every $n\in\bb{N}$. Note that in the case where either $\cl{A}$ or $\cl{B}$ is commutative, positivity implies complete positivity \cite[Theorem 3.11]{Pa} and \cite[Theorem 3.9]{Pa}.

\section{Quantum Probability Measures}\label{secqpm}
Let $H$ be a separable Hilbert space and $A$ be a second countable compact Hausdorff space. Recall that a $\cl{B}(H)$-valued \emph{quantum probability measure} \cite{ffp} on $A$ is a map $E:\mathfrak{A}_A\rightarrow \mathcal{B}(H)$, such that
		\begin{enumerate}
  
                \item[(i)] $E(\delta)\in\cl{B}(H)^+$, $ \delta\in\frak{A}_A$;
			\item [(ii)] $E(\emptyset)=0$ and $E(A)=I$;
			\item [(iii)] for any countable collection of pairwise disjoint sets $\{\delta_n\}_{n\in\mathbb{N}}\subseteq\mathfrak{A}_A$, we have $$E\left(\bigcup_{n=1}^\infty\delta_n\right)=\sum_{n=1}^\infty E(\delta_n),$$
			where the sum on the right hand side converges in the strong operator topology. 
			\end{enumerate}
Moreover, if $E(\delta\cap\theta)=E(\delta)E(\theta)$ whenever $\delta,\theta\in\mathfrak{A}_A$, we will call $E$ a spectral measure. We write $Q(A;H)$ for the set of all quantum probability measures on $A$ with values in $\cl{B}(H).$ By \cite[Proposition 4.5]{Pa} (see also \cite{btt}) there is a one to one correspondence between 
 the elements of $Q(A;H)$ and the unital completely positive maps from $C(A)$ to $\cl{B}(H)$, we will denote by $\phi_E$ the unique map that corresponds to $E\in Q(A;H)$ and satisfies the following formula.
 $$\langle \phi_E(f)\xi,\eta\rangle=\int_AfdE_{\xi,\eta}, \hspace{0.2cm}f\in C(A), \hspace{0.2cm}\xi,\eta\in H,$$
 where $E_{\xi,\eta}$ denotes the scalar Borel measure $E_{\xi,\eta}(\alpha)=\langle E(\alpha)\xi,\eta\rangle$, $\alpha\in\frak{A}_A.$ 
 In the sequel we will systematically use this identification to obtain metric characterizations between the newly introduced $\rho$ and $\delta$ distances and the completely bounded norm. 
\subsection{Total Variation Analogues}
Recall that for a measure $\mu\in M(A)$ the total variation of $\mu$ is defined as 
$$\|\mu\|_{tv}=\sup\left\{\sum_{i=1}^n\lvert\mu(\alpha_i)\rvert:\{\alpha_i\}_{i=1}^n\subseteq\frak{A}_A, \text{ partition of A} \right\}. $$
\begin{definition}
	Let $E\in Q(A;H)$. The total variation of $E$ is the quantity 
	
	$$\|E\|_{TV}=\operatorname{sup}\{\|E_{\xi,\eta}\|_{tv}:\xi,\eta\in H, \|\xi\|=\|\eta\|=1\}.$$ 
For $E,F\in Q(A;H)$ we define the \emph{$\rho$-distance} between $E$ and $F$ by letting $$\rho(E,F)=\|E-F\|_{TV}.$$
\end{definition}
    Notice that $\|E\|_{TV}\leq1$, since for any $\xi,\eta\in H$, we have that $\|E_{\xi,\eta}\|_{tv}\leq\|\xi\|\|\eta\|$, which consequently implies that the $\rho$-distance is bounded.

    \begin{proposition}\label{q_met}
	The $\rho$-distance is a metric on $Q(A;H)$.
\end{proposition}
\begin{proof}
	Let $E,F\in Q(A;H)$. Then $\rho(E,F)=0$ implies that  $\|E_{\xi,\eta}-F_{\xi,\eta}\|_{tv}=0$ for any pair of unit vectors $\xi,\eta\in H$. Since $\|\cdot\|_{tv}$ is a metric on the unit sphere of $M(A)$, one obtains that $\langle E(\alpha)\xi,\eta\rangle=\langle F(\alpha)\xi,\eta\rangle$ for every $\alpha\in\frak{A}_A$. Thus, $E=F$.
	\par The symmetry of $\rho$ is immediate from the symmetry of the total variation metric $\|\cdot\|_{tv}$ on $M(A)$.
	
	\par To show the triangle inequality, let $E,F,R\in Q(A;H)$; one obtains
	
	\begin{align*}
		\rho(E,F) & =\sup_{\|\xi\|=\|\eta\|=1}\|E_{\xi,\eta}-R_{\xi,\eta}+R_{\xi,\eta}-F_{\xi,\eta}\|_{tv}\\
		& \leq\sup_{\|\xi\|=\|\eta\|=1}\|E_{\xi,\eta}-R_{\xi,\eta}\|_{tv}+\sup_{\|\xi'\|=\|\eta'\|=1}\|R_{\xi',\eta'}-F_{\xi',\eta'}\|_{tv}\\
		& = \rho(E,R)+\rho(R,F).
	\end{align*}
\end{proof}
 Utilizing the characterization of quantum probability measures as unital completely positive maps, we show that one can evaluate the $\rho$-distance between two quantum probability measures by calculating the norm of the difference between their corresponding unital completely positive maps. 
\begin{proposition}\label{tv_op}
	If $E,F\in Q(A;H)$, then $\rho(E,F)=\|\phi_E-\phi_F\|$.
\end{proposition}
\begin{proof}
	For $E\in Q(A;H)$ and $\xi,\eta\in H$, we denote by $\phi_{E,\xi,\eta}:C(A)\rightarrow \bb{C}$ the functional given by $\phi_{E,\xi,\eta}(f)=\langle \phi_E(f)\xi,\eta\rangle$, $f\in C(A)$. 
 \par Let $\epsilon>0$. Then, by the definition of $\rho$, there exist unit vectors $\xi,\eta\in H$ such that 
	
	$$\rho(E,F)\leq\|E_{\xi,\eta}-F_{\xi,\eta}\|_{tv}+\epsilon.$$
	Now define the map $\psi:C(A)\rightarrow\bb{C}$ by letting
	
	$$\psi(f)=\int_Afd(E_{\xi,\eta}-F_{\xi,\eta}),\hspace{0.3cm}f\in C(A).$$
By the Riesz Representation Theorem, we have that $\psi$ is bounded and $\|\psi\|=\|E_{\xi,\eta}-F_{\xi,\eta}\|_{tv}$. On the other hand, we have that 
	\begin{align*}
		\|\phi_{E,\xi,\eta}-\phi_{F,\xi,\eta}\| & = \sup_{\|f\|=1}\Big\lvert\phi_{E,\xi,\eta}(f)-\phi_{F,\xi,\eta}(f)\Big\rvert\\
		& = \sup_{\|f\|=1}\Big\lvert\int_AfdE_{\xi,\eta}-\int_AfdF_{\xi,\eta}\Big\rvert = \|\psi\|.
	\end{align*}
As a result, one obtains $$\rho(E,F)\leq\|\phi_{E,\xi,\eta}-\phi_{F,\xi,\eta}\|+\epsilon\leq\|\phi_E-\phi_F\|+\epsilon,$$ and since $\epsilon$ was arbitrary we have $\rho(E,F)\leq\|\phi_E-\phi_F\|$.
\par For the reverse inequality, we observe that
\begin{align*}
	\|\phi_E-\phi_F\|  = & \sup_{\|f\|=1}\|\phi_E(f)-\phi_F(f)\|_{\cl{B}(H)}\\
	= & \sup_{\|f\|=1}\sup_{\|\xi\|=\|\eta\|=1}\Big\lvert\langle\phi_E(f)\xi,\eta\rangle-\langle\phi_F(f)\xi,\eta\rangle\Big\rvert\\
	= & \sup_{\|f\|=1}\sup_{\|\xi\|=\|\eta\|=1}\Big\lvert\int_AfdE_{\xi,\eta}-\int_AfdF_{\xi,\eta}\Big\rvert\\
	\leq & \sup_{\|f\|=1}\sup_{\|\xi\|=\|\eta\|=1}\|f\|_\infty\|E_{\xi,\eta}-F_{\xi,\eta}\|_{tv}\\
	= & \rho(E,F).
\end{align*}
Therefore $\|\phi_E-\phi_F\|=\rho(E,F)$.
\end{proof}
Next we introduce an alternative metric on $Q(A;H)$, which we call the $\delta$-metric and relate it to the $\rho$-metric.
\begin{definition}
    Let $E,F\in Q(A;H)$. The \emph{$\delta$-distance} between $E$ and $F$ is given by setting 
    $$\delta(E,F)=\sup_{\alpha\in\frak{A}_A}\|E(\alpha)-F(\alpha)\|.$$
\end{definition}
\begin{theorem}\label{2d=r}
    Let $E,F\in Q(A;H)$. Then $2\delta(E,F)=\rho(E,F)$.
\end{theorem}
\begin{proof}
    First we will prove that for any $E,F\in Q(A;H)$ we can evaluate $\rho(E,F)$ by taking supremums with $\eta=\xi$, that is 
    $$\rho (E,F)=\sup_{\|\xi\|=1}\|E_{\xi,\xi}-F_{\xi,\xi}\|.$$
    \par If $\xi,\eta\in H$ with $\|\xi\|=\|\eta\|=1$, then for any partition $\{\alpha_i\}_{i=1}^n\subseteq\frak{A}_A$ of $A$ one obtains 
    \begin{align*}
        \sum_{i=1}^n\lvert E_{\xi,\eta}(\alpha_i)-F_{\xi,\eta}(\alpha_i)\rvert & = \sum_{i=1}^n\lvert \langle (E(\alpha_i)-F(\alpha_i))\xi,\eta\rangle\rvert\\
        & \leq \sum_{i=1}^n\|E(\alpha_i)-F(\alpha_i)\|_{\cl{B}(H)}\|\xi\|\|\eta\|\\
        & = \sum_{i=1}^n\sup_{\|\zeta\|=1}\lvert\langle(E(\alpha_i)-F(\alpha_i))\zeta,\zeta\rangle\rvert\\
        & \leq \sup_{\|\zeta\|=1}\|E_{\zeta,\zeta}-F_{\zeta,\zeta}\|_{tv},
    \end{align*}
    where the second equality holds since $E(\alpha_i)-F(\alpha_i)$ is selfadjoint for every $i\in\{1,\dots,n\}.$
\par Recall that the total variation between two scalar probability measures $\mu,\nu\in M(A)$ is given by $\|\mu-\nu\|_{tv}=\frac{1}{2}\sup_{\alpha\in\frak{A}_A}\lvert\mu(\alpha)-\nu(\alpha)\rvert$. Now since for any unit vector $\xi\in H$, we have that $E_{\xi,\xi}$ and $F_{\xi,\xi}$ are scalar probability measures on $A$ it follows that $2\delta(E,F)=\rho(E,F)$.
\end{proof}
\subsection{A Bures-type distance}
The \emph{Bures distance} for completely positive maps was first introduced by Kretschmann, Schlingemann and Werner in \cite{werner} to generalize the Bures metric for positive functionals \cite{br}. It evaluates the distance between two completely positive maps based on the smallest norm difference of the bounded operators involved on their dilations. 
\par Let $\cl{A}$ be a C$^*$-algebra, $K$ Hilbert space and $\pi$ a *-representation of $\cl{A}$ on $K$. Given a completely positive map $\phi:\cl{A}\rightarrow\cl{B}(H)$, we say that $\pi$ \emph{dilates} $\phi$ if there exists an operator $V\in\cl{B}(H,K)$ such that  
\begin{align}\label{dilate}
    \phi(a)=V^*\pi(a)V,\hspace{0.3cm}a\in\cl{A}.
\end{align}
For simplicity, we will write $\phi=V^*\pi V$ whenever formula (\ref{dilate}) holds.   
The collection of all operators in $\cl{B}(H,K)$ that dilate $\phi$ will be denoted by  
$$S(\phi,\pi)=\{V\in\cl{B}(H,K):\phi=V^*\pi V\}.$$

\begin{definition}\label{bur_def}
    Let $\cl{A}$ be a C$^*$-algebra, $H$ a Hilbert space, $\phi_i:\cl{A}\rightarrow\cl{B}(H)$, $i=1,2$, be completely positive maps and $\pi$ a representation of $\cl{A}$ on some Hilbert space $K$. The $\pi$-distance between $\phi_1$ and $\phi_2$ is defined as $$\beta_\pi(\phi_1,\phi_2)=\inf\{\|V_1-V_2\|:V_i\in S(\phi_i,\pi)\}.$$ 
    If $S(\phi_i,\pi)=\emptyset$ for some $i=1,2$, set $\beta_\pi(\phi_1,\phi_2)=2$. The Bures distance between $\phi_1$ and $\phi_2$ is given by setting 
    $$\beta(\phi_1,\phi_2)=\inf_\pi\beta_\pi(\phi_1,\phi_2).$$
    
\end{definition}
In \cite{werner}, a continuity result for Stinespring's Theorem is provided, by showing that the topologies arising from the completely bounded norm $\|\cdot\|_{cb}$ and the Bures distance are equivalent. The theorem is captured by the following inequality 

$$\frac{\|\phi_1-\phi_2\|_{cb}}{\sqrt{\|\phi_1\|_{cb}}+\sqrt{\|\phi_2\|_{cb}}}\leq\beta(\phi_1,\phi_2)\leq\sqrt{\|\phi_1-\phi_2\|_{cb}}.$$
We define a Bures-type distance $\tilde{\beta}$ on $Q(A;H)$ and prove that for any $E,F\in Q(A;H)$ one has $\tilde{\beta}(E,F)=\beta(\phi_E,\phi_F)$. For $E\in Q(A;H)$ and a spectral measure $F\in Q(A;K)$ we say that $F$ \emph{dilates} $E$ if there exists $V\in\cl{B}(H,K)$ such that $$E(\alpha)=V^*F(\alpha)V, \hspace{0.3cm} \alpha\in\frak{A}_A,$$
and write $E=V^*F V$. Similarly to definition \ref{bur_def}, we write 
$$\tilde{S}(E,F)=\{V\in\cl{B}(H,K):E=V^*FV\}.$$

\begin{definition}
    Let $A$ be a compact Hausdorff space, $H$ a Hilbert space, $E_i\in Q(A;H)$, $i=1,2$ and $F$ a $\cl{B}(K)$-valued spectral measure on $A$. The $F$-distance between $E_1$ and $E_2$ is given by 

    $$\tilde{\beta}_F(E_1,E_2)=\inf\{\|V_1-V_2\|:V_i\in\tilde{S}(E_i,F)\}.$$
    If $\tilde{S}(E_i,F)=\emptyset$, then set $\tilde{\beta}_F(E_1,E_2)=2$. The Bures distance on $Q(A;H)$ is defined as
    $$\tilde{\beta}(E_1,E_2)=\inf_F\tilde{\beta}_F(E_1,E_2),$$
    where the infimum is taken over all spectral measures. 
\end{definition}

\begin{theorem}\label{bur_bur}
    Let $E_1,E_2\in Q(A;H)$. Then $\tilde{\beta}(E_1,E_2)=\beta(\phi_{E_1},\phi_{E_2})$.
\end{theorem}
\begin{proof}
Let $E_1,E_2\in Q(A;H)$ and $F\in Q(A;H)$ be a spectral measure. Take $V_i\in\tilde{S}(E_i,F)$, $i=1,2$, then by definition we have that $E_i=V_i^*FV_i$. Following the approximation arguments in \cite{btt}, we can extend $\phi_{E_i}$ to include the characteristic functions of elements in $\frak{A}_A$ such that $E_i(\alpha)=\phi_{E_i}(\chi_\alpha)$, for $\alpha\in\frak{A}_A$. As a result, one has $\phi_{E_i}=V_i^*\pi_FV_i$, $i=1,2$ which implies $V_i\in S(\phi_{E_i},\pi_F)$. Therefore $\tilde{S}(E_i,F)\subseteq S(\phi_{E_i},\pi_F)$ for $i=1,2$ and hence $\tilde{\beta}_F(E_1,E_2)\geq\beta_{\pi_F}(\phi_{E_1},\phi_{E_2})$ which consequently implies $\tilde{\beta}(E_1,E_2)\geq\beta(\phi_{E_1},\phi_{E_2})$.
\par Conversely, let $\pi$ be a representation of $C(A)$ on some Hilbert space $K$. Let $V_i\in S(\phi_{E_i},\pi)$, using the same argument as in the previous paragraph we have that $V_i\in\tilde{S}(E_i,F_\pi)$ where $F_\pi$ is the unique spectral measure corresponding to $\pi$ and hence $\tilde{S}(E_i,F_\pi)\supseteq S(\phi_{E_i},\pi)$ for $i=1,2$. Finally $\tilde{\beta}_{F_\pi}(E_1,E_2)\geq\beta_{\pi}(\phi_{E_1},\phi_{E_2})$, and thus $\tilde{\beta}(E_1,E_2)\geq\beta(\phi_{E_1},\phi_{E_2})$.
\end{proof} 
In Proposition \ref{tv_op} we identified the $\rho$-metric between two quantum probability measures with the operator norm of the corresponding unital completely positive maps. As a corollary of Theorem \ref{bur_bur}, we can translate the main result in \cite{werner} in terms of the metrics defined on this paper. Corollary \ref{naim_cont} can be viewed as a continuity result for Naimark's Dilation Theorem. 

\begin{corollary}\label{naim_cont}
    Let $E_1,E_2\in Q(A;H)$ Then the following holds 
    $$\frac{\rho(E,F)}{\sqrt{\|E_1\|_{TV}}+\sqrt{\|E_2\|_{TV}}}\leq\tilde{\beta}(E_1,E_2)\leq\sqrt{\rho(E,F)}.$$
    \end{corollary}

\subsection{Weak Topology} For the remainder of this section we introduce a weak mode of convergence for elements of $Q(A;H)$ and associate it with Arveson's BW topology as it was introduced in \cite{arveson-acta}. Recall that the BW topology on bounded sets of $\cl{B}(C(A),\cl{B}(H))$ is given as the point weak* topology induced by the duality $\cl{T}(H)^*\cong \cl{B}(H)$. 

\begin{definition}
	Let $(E^i)_{i\in I}$ be a net in $Q(A;H)$. We will say that $(E^i)_{i\in I}$ converges to some $E\in Q(A;H)$ in the set weak$^*$(SW) topology if $E^i(\alpha)$ converges to $E(\alpha)$ in the w$^*$ topology for every $\alpha\in\frak{A}_A$. 
\end{definition}

\begin{theorem}\label{ps_bw}
	Let $(E^i)_{i\in I}$ be a net in $Q(A;H)$ and $E\in Q(A;H)$. Then $(E^i)_{i\in I}$ converges to $E$ in the SW topology if and only if $(\phi_{E^i})_{i\in I}$ converges to $\phi_E$ in the BW topology.
\end{theorem}
\begin{proof}
	First assume that $E^i$ converges to $E$ in the SW topology, that is $E^i(\alpha)$ converges to $E(\alpha)$ in the w$^*$ topology, for every $\alpha\in\frak{A}_A$. This means that for all unit vectors $\xi,\eta$ we have $$\langle E^i(\alpha)\xi,\eta\rangle\rightarrow \langle E(\alpha)\xi,\eta\rangle.$$
Fix $f\in C(A)$. Since $A$ is metrizable, for any $n\in \bb{N}$ there exists a step function $f_n=\displaystyle\sum_{j=1}^{k_n}\lambda_{j,n}\chi_{\delta_{j,n}}$ with $\lambda_{j,n}\in\bb{C}$ and $\delta_{j,n}\in\frak{A}_A$ such that $\|f-g_n\|_\infty<\frac{1}{n}.$ Therefore  
	\begin{align*}
		\Big\lvert\int_AfdE^i_{\xi,\eta}-\int_AfdE_{\xi,\eta}\Big\rvert = & 	\Big\lvert\int_AfdE^i_{\xi,\eta} -\int_Ag_ndE^i_{\xi,\eta}+\int_Ag_ndE^i_{\xi,\eta}\\ -&
		 \int_Ag_ndE_{\xi,\eta}+\int_Ag_ndE_{\xi,\eta}-\int_AfdE_{\xi,\eta}\Big\rvert\\
		 \leq & \int_A\lvert f-g_n\rvert d\lvert E^i_{\xi,\eta}\rvert+\Big\lvert\int_Ag_ndE^i_{\xi,\eta}-\int_Ag_ndE_{\xi,\eta}\Big\rvert\\
		 + & \int_A\lvert g_n-f\rvert d\lvert E_{\xi,\eta}\rvert\\
		 \leq &  2\|f-g_n\|_\infty\|\xi\|\|\eta\|\\
    + & \Big\lvert\sum_{j=1}^{k_n}\lambda_{j,n}\langle (E^i(\delta_{j,n})-E(\delta_{j,n}))\xi,\eta\rangle\Big\rvert\\
		 \leq & \frac{2}{n}+  \Big\lvert\sum_{j=1}^{k_n}\lambda_{j,n}\langle (E^i(\delta_{j,n})-E(\delta_{j,n}))\xi,\eta\rangle\Big\rvert
	\end{align*}
	Let $n\in\bb{N}$, since $\langle E^i(\delta)\xi,\eta\rangle\rightarrow\langle E(\delta)\xi,\eta\rangle$ for any $\delta\in\frak{A}_A$ it follows that $$\sum_{j=1}^{k_n}\lambda_{j,n}\langle E^i(\delta_{j,n})\xi,\eta\rangle\rightarrow \sum_{j=1}^{k_n}\lambda_{j,n}\langle E(\delta_{j,n})\xi,\eta\rangle$$ and as a result $\int_Af dE^i_{\xi,\eta}\rightarrow\int_Af dE_{\xi,\eta}$ which implies that $ \phi_{E^i}$ converges to $\phi_E$ in the BW topology. 
	\par Conversely, suppose $\phi_{E^i}\rightarrow\phi_E$ in the BW topology, then since $(\phi_{E^i})_{i\in I}$ is a normed bounded net we have that for every $f\in C(A)$ and for any $\xi,\eta\in H$ $\langle \phi_{E^i}(f)\xi,\eta\rangle\rightarrow\langle\phi_E(f)\xi,\eta\rangle$ implies that 
	
	$$\int_AfdE^i_{\xi,\eta}\rightarrow\int_Af dE_{\xi,\eta}.$$
	\par We will show that $E^i_{\xi,\eta}(\alpha)\rightarrow E_{\xi,\eta}(\alpha)$, for all $\alpha\in\frak{A}_A$. Since $\frak{A}_A$ is generated by the open sets and $E^i_{\xi,\eta}\in M(A)$ for every $i\in I$ and $\xi,\eta\in H$, it suffices to prove our claim for some $\alpha=B_r(a_0)$, where $B_r(a_0)$ is the closed ball of radius $r>0$ centered at $a_0\in A$. For each $n\in\bb{N}$ let $U_n=B_{r+\frac{1}{n}}(a_0)$. Then by Urysohn's lemma there exists continuous functions $f_n$ with values in $[0,1]$, converging pointwise to $\chi_\alpha$ such that $f_n(a)=1$, whenever $a\in \alpha$ and $f_n(a)=0$ whenever $a\not\in\overline{U_n}$. Therefore by Lebesgue Dominated Convergence Theorem, it follows that  for every $i\in I$ we have $\int_Af_ndE^i_{\xi,\eta}$ converges to $\int_A\chi_\alpha dE^i_{\xi,\eta}=\langle E(\alpha)\xi,\eta\rangle$ as $n\rightarrow\infty$ and $\int_Af_ndE_{\xi,\eta}\rightarrow\int_A\chi_\alpha dE_{\xi,\eta}=\langle E(\alpha)\xi,\eta\rangle.$ As a result one obtains 
	
\begin{align*}	
	 \Big\lvert E^i_{\xi,\eta}(\alpha)-E_{\xi,\eta}(\alpha)\Big\rvert & =	\Big\rvert E^i_{\xi,\eta}(\alpha)-\int_Af_ndE^i_{\xi,\eta}+\int_Af_ndE^i_{\xi,\eta}-\int_Af_ndE_{\xi,\eta}\\ & +\int_Af_ndE_{\xi,\eta}-E_{\xi,\eta}(\alpha)\Big\rvert\\
	 & \leq \Big\rvert E^i_{\xi,\eta}(\alpha)-\int_Af_ndE^i_{\xi,\eta}\Big\rvert+\Big\lvert\int_Af_ndE^i_{\xi,\eta}-\int_Af_ndE_{\xi,\eta}\Big\rvert\\ &+\Big\lvert\int_Af_ndE_{\xi,\eta}-E_{\xi,\eta}(\alpha)\Big\rvert,
\end{align*}
 which implies that $E^i_{\xi,\eta}(\alpha)\rightarrow E_{\xi,\eta}(\alpha)$ for every $ \alpha\in\frak{A}_A$. Thus $\langle E^i(\alpha)\xi,\eta\rangle\rightarrow\langle E(\alpha)\xi,\eta\rangle$ implies $\operatorname{Tr}(E^i(\alpha)(\xi\otimes\eta))\rightarrow\operatorname{Tr}(E(\alpha)(\xi\otimes\eta)),$ for every $\xi,\eta\in H$, which means that $E^i(\alpha)$ converges to $E(\alpha)$ in the w$^*$ topology for every $\alpha\in\frak{A}_A$.
\end{proof}
\begin{remark}
    The space $Q(A;H)$ equipped with the SW topology is homeomorphic to $\operatorname{UCP}(C(A),\cl{B}(H))$ equipped with the BW topoology and hence by \cite[Theorem 7.4.]{Pa} it is compact.
\end{remark}

\section{Operator-Valued Information Channels}
Operator-valued information channels are non-commutative analogues of classical information channels, firstly initiated to define the newly introduced measurable no-signalling quantum correlations and their subclasses in \cite{btt}. In this section we extend the modes of convergence for scalar valued channels studied in \cite{jacobs,kakihara} by utilizing the results introduced in section \ref{secqpm}. 

\begin{definition}
    Let $A$, $X$ be compact Hausdorff spaces and $H$ a separable Hilbert space. A $\cl{B}(H)$-valued information channel from $X$ to $A$ is a family of quantum probability measures $E=(E(\cdot|x))_{x\in X}$ such that the function $x\mapsto E(\alpha|x)$ is weakly measurable for any $\alpha\in\frak{A}_A$. 
\end{definition}
We will denote by $\frak{C}(A,X;H)$ the set of all $\cl{B}(H)$-valued information channels from $X$ to $A$.
A result in \cite{btt} provides a characterization of operator-valued information channels as measurable families of unital completely positive maps. Similarly with section \ref{secqpm}, we will use this result to provide metric and topological characterizations of the newly introduced topologies for operator-valued information channels. 
\par We denote by  $\cl{L}^\infty(X,\cl{B}(H))$ the space of all bounded w$^*$-measurable functions from $X$ to $\cl{B}(H)$, and note that it is a unital $C^*$-algebra when equipped with pointwise operations, the supremum norm and involution defined pointwise. Moreover, for every $x\in X$ we denote by $E_x$ the quantum probability measure $E(\cdot|x)$ and $\phi_{E_x}$ the corresponding unital completely positive map.
\begin{theorem}\label{th_Borel}
		Let $A,X$ be compact Hausdorff spaces, $A$ be also second countable and $E=(E_x)_{x \in X}$ be a $\cl{B}(H)$-valued information channel from $X$ to $A$. Then the map $\Phi_E : C(A)\rightarrow\cl{L}^\infty(X,\cl{B}(H))$, given by 
		$$\Phi_E(f)(x) = \phi_{E_x}(f), \hspace{0.3cm} f\in C(A), x\in X,$$
		is well-defined, unital and completely positive.

Conversely, if $\Phi : C(A)\rightarrow \cl{L}^\infty(X,\cl{B}(H))$  is a unital completely positive map
there exists $E\in\frak{C}(A,X;H)$ such that $\Phi = \Phi_E$.
	\end{theorem}
 \subsection{The $\tilde{\rho}$-distance}
\begin{theorem}
	The set $\frak{C}(A,X;H)$ equipped with the function $\tilde{\rho}:\frak{C}(A,X;H)\times \frak{C}(A,X;H)\rightarrow\bb{R}$, given by 
	$$\tilde{\rho}(E,F)=\sup_{x\in X}\rho(E_x,F_x),$$
	is a metric space.
\end{theorem}
\begin{proof}
	Let $E,F\in \frak{C}(A,X;H)$ and suppose that $\tilde{\rho}(E,F)=0$; this means that $\sup_{x\in X}\rho(E_x,F_x)=0$, for every $ x\in X$ which implies that $E_x=F_x$ for each $x\in X$ and hence $ E=F$. 
	\par Moreover, note that the symmetry of $\tilde{\rho}$ is straightforward from the symmetry of $\rho$.
	\par Finally since $\rho$ satisfies the triangle inequality, we have that for any $Q\in \frak{C}(A,X;H)$ the following holds:
	
	\begin{align*}
	\tilde\rho(E,F) & =\sup_{x\in X}\rho(E_x,F_x) \leq \sup_{x\in X}\rho(E_x,Q_x)+\rho(Q_x,F_x)\\
	& \leq\sup_{x\in X}\rho(E_x,Q_x)+\sup_{y\in X}\rho(Q_y,F_y) = \tilde\rho(E,Q)+\tilde\rho(Q,F).
	\end{align*}
\end{proof}
We can obtain a similar result to Proposition \ref{tv_op} for operator-valued information channels.
\begin{proposition}
	If $E,F\in \frak{C}(A,X;H)$, then $\tilde{\rho}(E,F)=\|\Phi_E-\Phi_F\|$.
\end{proposition}
\begin{proof}
	Let $E,F\in \frak{C}(A,X;H)$. Then the following holds 
	\begin{align*}
		\|\Phi_E-\Phi_F\| & =\sup_{\|f\|=1}\|\Phi_E(f)-\Phi_F(f)\|_{\cl{L}^\infty(X,\cl{B}(H))}\\
		& =\sup_{\|f\|=1}\sup_{x\in X}\|\phi_{E_x}(f)-\phi_{F_x}(f)\|_{\cl{B}(H)}\\
		& =\sup_{\|f\|=1}\sup_{x\in X}\sup_{\|\xi\|=\|\eta\|=1}\lvert\langle\phi_{E_x}(f)\xi,\eta\rangle-\langle\phi_{F_x}(f)\xi,\eta\rangle\rvert\\
		& = \sup_{\|f\|=1}\sup_{x\in X}\sup_{\|\xi\|=\|\eta\|=1}\Big\lvert\int_Afd(E_{\xi,\eta}(\cdot|x)-F_{\xi,\eta}(\cdot|x))\Big\rvert\\
		& \leq\sup_{\|f\|=1}\sup_{x\in X}\sup_{\|\xi\|=\|\eta\|=1}\|f\|_\infty\|E_{\xi,\eta}(\cdot|x)-F_{\xi,\eta}(\cdot|x)\|_{tv}\\
		& =\sup_{x\in X}\rho(E_x,F_x)=\tilde\rho(E,F).
	\end{align*}
For the reverse inequality, let $\epsilon>0$, then there exist $x_0\in X$ such that 
$$\tilde\rho(E,F)<\rho(E_{x_0},F_{x_0})+\frac{\epsilon}{3}=\|\phi_{E_{x_0}}-\phi_{F_{x_0}}\|+\frac{\epsilon}{3}.$$
Furthermore there exists $f\in C(A)$ with $\|f\|_\infty=1$ such that 
$$\|\phi_{E{x_0}}-\phi_{F_{x_0}}\|\leq\|\phi_{E{x_0}}(f)-\phi_{F_{x_0}}(f)\|+\frac{\epsilon}{3},$$
and unit vectors $\xi,\eta\in H$, satisfying 
$$\|\phi_{E{x_0}}(f)-\phi_{F_{x_0}}(f)\|\leq\lvert\langle\phi_{E_{x_0}}(f)\xi,\eta\rangle-\langle \phi_{F_{x_0}}(f)\xi,\eta\rangle\rvert+\frac{\epsilon}{3}.$$
Finally one obtains $$\tilde\rho(E,F)\leq\lvert\langle\phi_{E_{x_0}}(f)\xi,\eta\rangle-\langle \phi_{F_{x_0}}(f)\xi,\eta\rangle\rvert+\epsilon\leq\|\Phi_E-\Phi_F\|+\epsilon$$
and since $\epsilon$ was chosen arbitrarily $\tilde\rho(E,F)\leq\|\Phi_E-\Phi_F\|$.
\end{proof}

\subsection{Weak Topology}
Following the same route as in Section \ref{secqpm}, we define a weak topology on $\frak{C}(A,X;H)$. 

\begin{definition} Let $(E^i)_{i\in I}$ be a net in $\frak{C}(A,X;H)$ and $E\in\frak{C}(A,X;H)$. We say that $(E^i)_{i\in I}$ converges to $E$ in the point set weak* (pSW) topology if $(E^i(\cdot|x))_{i\in I}$ converges to $E(\cdot|x)$ in the SW topology for every $x\in X$.
\end{definition}
Utilizing the correspondence provided by Theorem \ref{th_Borel}, we will associate the pSW topology with a topology on the set $\operatorname{UCP}(C(A),\cl{L}^\infty(X,\cl{B}(H)))$ called the point BW (pBW) topology.  
\begin{definition}
	Let $(\Phi^i)_{i\in I}$ be a net in $\operatorname{UCP}(C(A),\cl{L}^\infty(X,\cl{B}(H)))$. We will say that $(\Phi^i)_{i\in I}$ converges to some $\Phi\in\operatorname{UCP}(C(A),\cl{L}^\infty(X,\cl{B}(H))$ in the point BW (pBW) topology if $(\Phi_x^i)_{i\in I}$ converges to $\Phi_x$ in the BW topology for every $x\in X$.
\end{definition}
\begin{theorem}\label{psw_pbw}
Let $(E^i)_{i\in I}$ be a net in $\frak{C}(A,X;H)$ and $E\in \frak{C}(A,X;H)$. Then $(\Phi_{E_i})_{i\in I}$ converges to $\Phi_E$ in the pBW topology if and only if $(E^i)_{i\in I}$ converges to $E$ in the pSW topology.
\end{theorem}
\begin{proof}
	Let $\Phi_{E^i}\rightarrow\Phi_E$ in the pBW topology. Then $\Phi_{E^i_x}\rightarrow\Phi_{E_x}$ in the BW topology for every $ x\in X$. Thus by Theorem \ref{ps_bw}, one obtains $E^i_x\rightarrow E_x$ in the SW topology for every $ x\in X$, which implies that $E^i\rightarrow E$ in the pSW topology.
	
	\par Conversely, suppose $E^i\rightarrow E$ in the pSW topology, then $E^i_x\rightarrow E_x$ in the SW topology for every $ x\in X$. Again by Theorem  \ref{ps_bw} it follows that $\Phi_{E^i_x}\rightarrow\Phi_{E_x}$ in the BW topology for every $ x\in X$ which means that $\Phi_{E^i}\rightarrow\phi_E$ in the pBW topology. 
\end{proof} 

\begin{theorem}\label{ov_seq}
	The topological space $(\operatorname{UCP}(C(A),\cl{L}^\infty(X,\cl{B}(H))),pBW)$ is sequentially compact.
\end{theorem}
\begin{proof}
	For simplicity, we denote $\Delta=\operatorname{UCP}(C(A),\cl{L}^\infty(X,\cl{B}(H)))$ and $\Gamma=\operatorname{UCP}(C(A),\cl{B}(H))$. Consider the topological space $(\Gamma,BW)$ and construct the product space $\Gamma^X=\prod_{x\in X}\Gamma_x$, where $\Gamma_x=\Gamma$, for all $x\in X$ equipped with the product topology $\tau$. We define the map $\Lambda: (\Delta,pBW)\rightarrow(\Gamma^X,\tau)$ by 
	
	$$\Lambda(\Phi)=(\phi_x)_{x\in X},$$
	where $\phi_x=\Phi(\cdot)(x), x\in X$. Notice that $\Lambda$ is injective, since for any $\Phi,\Psi\in\Delta$ we have 
	$$\Lambda(\Phi)=\Lambda(\Psi)\iff(\phi_x)_{x\in X}=(\psi_x)_{x\in X}\iff\phi_x=\psi_x, \forall x\in X,$$
 which is equivalent with $\Phi=\Psi$.
	Let $(\Phi^i)_{i\in I}$ be a net in $\Delta$ and a $(\phi_x)_{x\in X}$ a family of unital completely positive maps from $C(A)$ to $\cl{B}(H)$ (not neccessarily measurable) and consider $\Phi(\cdot)(x)=\phi_x(\cdot)(f)$, $x\in X$. Suppose that $(\Phi^i)_{i\in I}$ converges to $\Phi$ in the pBW topology, that means $\phi^i_x\rightarrow\phi_x$ in the BW topology $\forall x\in X$, which is equivalent with $(\phi^i_x)_{x\in X}\rightarrow(\phi_x)_{x\in X}$ in the product topology. Therefore $(\Delta,pBW)$ is homeomorphic with $(\Lambda(\Delta),\tau)$. 
	\par It follows from Tychonoff's Theorem that $\Gamma^X$ is compact and hence to conclude the proof it suffices to show that $(\Delta,pBW)$ is sequentially closed. Let $(\Phi_n)_{n\in\bb{N}}$ be a sequence in $\Delta$, then for any $f\in C(A)$ and $\xi,\eta$ one obtains that $g_n(x)=\langle\Phi_n(f)(x)\xi,\eta\rangle$, $n\in\bb{N}$ is a sequence of measurable functions and therefore any pointwise limit of those is also measurable, which means that $(\Delta,pBW)$ is sequentially closed.
\end{proof}
\begin{corollary}
The topological space $(\frak{C}(A,X;H),pSW)$ is sequentially compact. 
\end{corollary}
\begin{proof}
	By Theorem \ref{psw_pbw} we have that $(\operatorname{UCP}(C(A),\cl{L}^\infty(X,\cl{B}(H))) pBW)$ is homeomorphic to $(\frak{C}(A,X;H),pSW)$ and since $\operatorname{UCP}(C(A),\cl{L}^\infty(X,\cl{B}(H)))$ equipped with the pBW is sequentially compact so is $(\frak{C}(A,X;H),pSW)$.
\end{proof}
\subsection{Operator-valued Information Channels mod $\mu$}

Throughout this section, we will fix a probability measure $\mu$ on $(X,\frak{A}_X)$ and study measurable families of unital completely positive maps defined $\mu$-almost everywhere. From now on $A$ is assumed to be second countable.

We define an equivalence relation on $\frak{C}(A,X;H)$ in the following way: We say that $E$, $F\in\frak{C}(A,X;H)$ are $\mu$-equivalent and write $E\sim_\mu F$ if for every $ \alpha\in \frak{A}_A$ we have $$E(\alpha|x)=F(\alpha|x)\hspace{0.4cm}\mu\text{-a.e.}.$$
Similarly we can define an equivalence relation on the set of all elements of $\operatorname{UCP}(C(A),\cl{L}^\infty(X,\cl{B}(H)))$ in the following way:
Let $\Phi$, $\Psi$ be unital completely positive maps from $C(A)$ to $\cl{L}^\infty(X,\cl{B}(H))$. We say that $\Phi$ and $\Psi$ are $\mu$-equivalent and write $\Phi\sim_\mu \Psi$ if for every $ f\in C(A)$ we have $$\Phi(f)(x)=\Psi(f)(x), \hspace{0.4cm}\mu\text{-a.e.}$$
Once again using the characterization obtained in Theorem \ref{th_Borel}, we show that the two equivalence relations are compatible.
\begin{proposition}\label{equi}
    Let $E,F\in\frak{C}(A,X;H)$ and $\mu$ a Borel measure on $X$. Then $E\sim_\mu F$ if and only if $\Phi_E\sim_\mu \Phi_F$.
    \end{proposition}
    \begin{proof}
        Let $E\sim_\mu F$, that is for every  $\alpha\in\frak{A}_A$ we have that $E(\alpha|x)=F(\alpha|x)$ $\mu$-almost everywhere. Thus for every $ \xi,\eta\in H$ one has
        $$E_{\xi,\eta}(\alpha|x)=F_{\xi,\eta}(\alpha|x),\hspace{0.2cm}\mu\text{-a.e.},$$
        since $A$ is assumed to be second countable we have that for every $f\in C(A)$ 
        $$\int_Af(a)dE_{\xi,\eta}(a|x)=\int_Af(a)dF_{\xi,\eta}(a|x)\hspace{0.2cm}\mu\text{-a.e.}$$
       which implies that $\Phi_E\sim_\mu\Phi_F$.
       \par Conversely, assume $\Phi_E\sim_\mu\Phi_F$ and let $\alpha\in\frak{A}_A$. Then by approximating $\chi_\alpha$ pointwise with $(f_n)_{n\in\bb{N}}$, we have that for each $n\in\bb{N}$, $\Phi_E(f_n)=\Phi_F(f_n)$ $\mu$-almost eveywhere. Hence there exists a family $\{U_n\}_{n\in\bb{N}}$ of $\mu$-negligible sets satisfying $$\Phi_E(f_n)(x)=\Phi_F(f_n)(x),\hspace{0.2cm}x\in X\setminus U_n.$$
       Consequently, since $f_n$ converges to $\chi_\alpha$ by applying the Dominated Convergence Theorem for each $x\in X\setminus \cup_nU_n$ we have that $\Phi_E(f_n)$ converges to $\Phi_E(\chi_\alpha)$ and $\Phi_F(f_n)$ converges to $\Phi_F(\chi_\alpha)$ $\mu$-almost everywhere. Finally, by $\sigma$-additivity we have that $\mu(\cup_nU_n)=0$ and therefore 
       $$\Phi_E(\chi_\alpha)(x)=\lim_{n\rightarrow\infty}\Phi_E(f_n)(x)=\lim_{n\rightarrow\infty}\Phi_F(f_n)(x)=\Phi_F(\chi_\alpha)(x), \hspace{0.2cm}x\in X\setminus \cup_n U_n.$$ 
    \end{proof}
Next we define $\frak{C}_\mu(A,X;H)$ to be the set of all $\sim_\mu$ equivalent classes of $\frak{C}(A,X;H)$. Elements of $\frak{C}_\mu(A,X;H)$ will be called \emph{$\cl{B}(H)$-valued information channels mod $\mu$}. To avoid confusion we will denote elements of $\frak{C}_\mu(A,X;H)$ by $\dot{E}$ and $E\in\frak{C}(A,X;H)$ for a specific representative of the equivalence class $\dot{E}$.
\par 
In a similar fashion to Theorem \ref{th_Borel} we provide a characterization theorem for operator-valued information channels mod $\mu$. Before proceeding with the result, we gather some required preliminaries regarding measurable functions with values on Banach spaces. 
\par Let $(X,\frak{X},\mu)$ be a $\sigma$-finite measure space and $Y$ be a Banach space. Consider the set of all  $(\frak{X},\frak{A}_Y)$-measurable functions $f:X\rightarrow Y$ and note that since the norm is a continuous function, the function $x\mapsto \|f(x)\|_Y$ is $(\frak{X},\frak{A}_{\bb{R}})$-measurable. For $1\leq p\leq\infty$, let $L^p(X,Y)$ be the space of all $(\frak{X},\frak{A}_Y)$-measurable functions defined up to $\mu$-negligible sets, such that  $$\|f\|_p=\Big(\int_X\|f(x)\|^p_Yd\mu(x)\Big)^{1/p}<\infty.$$
The space $L^p(X,Y)$ equipped with $\|\cdot\|_p$  is called the Bochner $L^p$ space (for more details see \cite{du}). In the case where $Y=Z^*$ for some separable Banach space $Z$, we say that a function $f:X\mapsto Z^*$ is w$^*$-measurable if for all $z\in Z$ the function $x\mapsto f(x)(z)$ is measurable. Notice that since $Z$ is separable, the w$^*$-measurability of $f$ implies the measurability of the function $x\mapsto\|f(x)\|_{Z^*}$. We denote by $L_\sigma^\infty(X,Z^*)$ the set of all w$^*$-measurable functions $f:X\rightarrow Z^*$ such that the quantity $$\|f\|_{\infty}:=\operatorname{ess sup}_{x\in X}\|f(x)\|_{Z^*},$$
is finite. By Dunford-Pettis Theorem \cite{dp}, one has

\begin{align}
    L_\sigma^\infty(X,Z^*)\cong L^1(X,Z)^*\cong (L^1(X)\widehat{\otimes}Z)^*,
\end{align}
where $\hat{\otimes}$ stands for the projective tensor product. For the definition of projective tensor products and the latter equivalence we refer to \cite{ryan}. 

\par In the context of this paper we shall restrict ourselves to the case where $Z=\cl{T}(H)$ and thus $Z^*=\cl{B}(H)$. Note that in this case, the space $L_\sigma^\infty(X,\cl{B}(H))$ is a von Neumann algebra when equipped with pointwise product and involution \cite{sakai}. Moreover, we will denote by $q:\cl{L}^\infty(X,\cl{B}(H))\rightarrow L_\sigma^\infty(X,\mu,\cl{B}(H))$ the natural surjection map and note that it is a *-homomorphism and hence a unital completely positive map.
\begin{theorem}\label{mod_mu}
    For any $\dot{E}\in\frak{C}_\mu(A,X;H)$ there exists a unital completely positive map $\Phi_{\dot{E}}:C(A)\rightarrow L_\sigma^\infty(X,\mu,\cl{B}(H))$ such that for every $ f\in C(A)$, and $ \xi,\eta\in H$ we have 
    $$\langle\Phi_{\dot{E}}(f)(x)\xi,\eta\rangle=\int_Af(a)d\dot{E}_{\xi,\eta}(a|x),\hspace{0.4cm} \mu\text{-a.e.}$$
    Conversely, for any unital completely positive map $\Phi:C(A)\rightarrow L_\sigma^\infty(X,\mu,\cl{B}(H))$  there exists a unique $\dot{E}\in\frak{C}_\mu(A,X,\cl{B}(H))$ such that $\Phi=\Phi_{\dot{E}}$.
\end{theorem}
\begin{proof}
    Let $\dot{E}\in\frak{C}_\mu(A,X;H)$ and consider a representative $E=(E(\cdot|x))_{x\in X}$ of the equivalence class $\dot{E}$. After applying Theorem \ref{th_Borel} we obtain a $\mu$-measurable family $(\phi_x)_{x\in X}$ of unital completely positive maps $\phi_x:C(A)\rightarrow \cl{B}(H)$, $x\in X$. Following the arguments of Theorem \ref{th_Borel} one obtains a unital completely positive map $\Phi_E:C(A)\rightarrow \cl{L}^\infty(X,\cl{B}(H))$ such that, for every $ f\in C(A)$ and $\xi,\eta\in H$ one has 
     $$\langle\Phi_E(f)(x)\xi,\eta\rangle=\int_Af(a)dE_{\xi,\eta}(a|x), \hspace{0.3cm}x\in X.$$
     Now for $f\in C(A)$, set $\Phi_{\dot{E}}(f)=q(\Phi_E(f))$ . Since $q$ is a unital completely positive map, it follows that $\Phi_{\dot{E}}:C(A)\rightarrow L_\sigma^\infty(X,\mu,\cl{B}(H))$ is unital and completely positive. 
     \par It remains to show that the construction of $\Phi_{\dot{E}}$ does not depend on the representative $E$. Assume that $F$ is another representative of the equivalence class $\dot{E}$, which means $E\sim_\mu F$. Hence by Proposition \ref{equi} one obtains $\Phi_E\sim_\mu\Phi_F$ and therefore for every $f\in C(A)$ we have $\Phi_E(f)=\Phi_F(f)$ $\mu$-\text{a.e.} which is equivalent to $q(\Phi_E(f))=q(\Phi_F(f))$.
     \par Conversely, let $\Phi:C(A)\rightarrow L^\infty_\sigma(X,\mu,\cl{B}(H))$ be a unital completely positive map, then by \cite[Theorem 4.6]{btt} there exists a Hilbert space $K$, a measurable family $(\pi_x)_{x\in X}$ of $^*$-representations of $C(A)$ on $K$ and an isometry $V:H\rightarrow K$ such that for every $f\in C(A)$
     \begin{align}\label{m_dil}
         \Phi(f)(x)=V^*\pi_x(f)V,\hspace{0.2cm}\mu\text{-a.e.}
    \end{align}
    Now by applying Theorem \ref{th_Borel} to the family $(\pi_x)_{x\in X}$ we obtain a spectral-valued information channel $\tilde{E}\in\frak{C}(A,X;H)$ such that $E(\alpha|x)=\pi_x(\chi_\alpha)$ for $x\in X$ and $\alpha\in\frak{A}_A$. By setting $E=V^*\tilde{E}V$ we obtain an operator-valued information channel $E\in\frak{C}(A,X;H)$ such that $$\Phi(f)(x)=\int_Af(a)dE_{\xi,\eta}(a|x),\hspace{0.2cm}\mu\text{-a.e.}$$
\end{proof}
The characterization provided in theorem \ref{mod_mu} combined with \cite[Theorem 4.6]{btt} allows us to formulate the following Naimark's Dilation type result for operator-valued information channels mod $\mu$.
\begin{corollary}
    For any $\dot{E}\in\frak{C}_\mu(A,X;H)$ there exists a separable Hilbert space $K$, an isometry $V:H\rightarrow K$ and a projection-valued $\dot{F}\in\frak{C}_\mu(A,X;K)$ such that for every $\alpha\in\frak{A}_A$

    $$\dot{E}(\alpha|x)=V^*\dot{F}(\alpha|x)V,\hspace{0.3cm}\mu\text{-a.e.}$$
\end{corollary}
For the remaining of this subsection, we equip $\frak{C}_\mu(A,X;H)$ with a topology induced by the convergence of their corresponding maps in the BW topology defined on $\cl{B}(C(A),L_\sigma^\infty(X,\mu,\cl{B}(H)))$.
\par First notice that the space $L^\infty_\sigma(X,\mu,\cl{B}(H))$ has a predual via the identification $$L^\infty_\sigma(X,\mu,\cl{B}(H))\simeq (L^1(X,\mu,\cl{T}(H)))^*\simeq(L^1(X,\mu)\hat\otimes\cl{T}(H))^*.$$
Therefore we can equip $\cl{B}(C(A),L_\sigma^\infty(X,\mu,\cl{B}(H)))$ with the BW topology and use the identification to induce a topology on $\frak{C}_\mu(A,X;H)$.
\begin{definition}
    Let $(\dot{E}^i)_{i\in I}$ be a net in $\frak{C}_\mu(A,X;H)$. We will say that $(E^i)_{i\in I}$ converges to some $E\in\frak{C}_\mu(A,X;H)$ if $(\Phi_{\dot{E}^i})_{i\in I}$ converges to $\Phi_{\dot{E}}$ in the BW topology.
\end{definition}
Recall, that in the case of bounded nets the BW topology is equivalent with saying that $\Phi_{\dot{E}^i}$ converges to $\Phi_{\dot{E}}$ if and only if $\Phi_{\dot{E}^i}(f)$ converges weakly to $\Phi_{\dot{E}}(f)$, for every $f\in C(A)$. Since maps that correspond to operator-valued information channels mod $\mu$ are unital and completely positive and hence have norm 1, from now on we shall use this latter formulation of the BW topology for bounded nets. 
\par For a concrete description of the convergence between operator-valued information channels mod $\mu$, notice that $\Phi_{\dot{E}^i}\xrightarrow{BW}\Phi_{\dot{E}}$ is equivalent with 
\begin{align}\label{BW}
    \int_X\Phi_{\dot{E}^i}(f)(x)(\omega(x)\otimes\xi\eta^*)d\mu(x)\rightarrow\int_X\Phi_{\dot{E}}(f)(x)(\omega(x)\otimes\xi\eta^*)d\mu(x),
\end{align}
for all $f\in C(A)$, $\omega\in L^1(X,\mu)$ and $\xi,\eta\in H$. Now by setting $f=\chi_\alpha$ and  $\omega=\chi_\beta$ for some $\alpha\in\frak{A}_A$ and $\beta\in\frak{A}_X$,  Eq. \ref{BW} has the following form 
\begin{align}
    \int_\beta\langle\dot{E}^i(\alpha|x)\xi,\eta\rangle d\mu(x)\rightarrow \int_\beta\langle \dot{E}(\alpha|x)\xi,\eta\rangle d\mu(x).
\end{align}
\begin{theorem}
    The space $\frak{C}_\mu(A,X;H)$ equipped with the BW topology is compact. 
\end{theorem}
\begin{proof}
    By the definition of the BW topology the statement is equivalent to showing that the topological space $\Delta=\operatorname{UCP}(C(A),L^\infty_\sigma(X,\mu,\cl{B}(H))$ is BW compact. First note that the closed unit ball of $\Gamma=\cl{B}(C(A),L_\sigma^\infty(X,\mu,\cl{B}(H))$ is BW-compact and hence it suffices to show that $\Delta$ is closed. Consider a net $(\Phi_\lambda)_{\lambda\in \Lambda}\subseteq\Delta$ that converges to some $\Phi\in\Gamma$. Since $\Phi_\lambda$ is unital and the closed unit ball of $\Gamma$ is BW-compact we have that $\|\Phi\|_{cb}\leq1$. Now since the domain of $\Phi$ is commutative it suffices to show that $\Phi$ is positive. Let $f\in C(A)^+$ with $\|f\|\leq1$, then $(\Phi_\lambda(f))_{\lambda\in \Lambda}\subseteq(L^\infty(X,\mu,\cl{B}(H))^+)_1$ is converging to $\Phi(f)$ in the $w^*$ topology. Since $L^\infty_\sigma(X,\mu,\cl{B}(H))$ is a von Neumann algebra it follows that $(L^\infty_\sigma (X,\mu,\cl{B}(H))^+)_1$ is $w^*$-closed and hence $\Phi(f)$ is positive.
\end{proof}
\subsection*{Acknowledgements}  The author would like to thank Ivan Todorov and Lyudmila Turowska for their valuable suggestions and comments. The author was supported by NSF grants CCF-2115071 and DMS-2154459. This work was supported by the University of Delaware through the UNIDEL fellowship and the Department of Mathematical Sciences through the Summer fellowship.


\begin{thebibliography}{99}


\bibitem{arveson-acta}
\textsc{W. B. Arveson}, 
{\it Subalgebras of C*-algebras}, 
{\rm Acta Math. 123 (1969), 141-224}. 

\bibitem{btt}
\textsc{G. Baziotis, I. G. Todorov and L. Turowska},
{\it Measurable no-signalling correlations},
{\rm preprint, 2024.}

\bibitem{br}
{\sc D. Bures},
{\it An extension of Kakutani's theorem on infinite product measures to the tensor product of semifinite w$^*$-Algebras},
{\rm Trans. Amer. Math. Soc. 135 (1969) 199.}

\bibitem{du}
{\sc J. Diestel and J. J. Uhl},
{\it Vector Measures},
{\rm Mathematical Surveys, vol. 15, American Mathematical Society, 1977.}

\bibitem{dp}
{\sc N. Dunford and B. J. Pettis}, 
{\it Linear operations on summable functions},
{\rm Trans. Am. Math. Soc. 47 (1940), p. 323-392}.

\bibitem{ffp}
{\sc D. Farenick, R. Floricel and S. Plosker}, 
{\it Approximately clean quantum probability measures}, 
{\rm J. Math. Phys. 54 (2013), no. 5, 052201, 15 pp.}

\bibitem{jacobs}
{\sc K. Jacobs},
{\it Die \"Ubertragung diskreter Informationen durch periodische und fastperiodische Kan\"ale},
{\rm Math. Annalen 137 (1959), 125-135}

\bibitem{kakihara}
{\sc Y. Kakihara}, 
{\it Abstract methods in information theory},
{\rm World Scientific Publishing, Hackensack, 2016}.

\bibitem{kieffer}
{\sc J. K. Kieffer},
{\it Some topologies on the set of discrete stationary channels}, 
{\rm Pacific J. Math. 105 (1983), 359-385.}

\bibitem{werner}
{\sc D. Krestchmann, D. Schlingemann and R. Werner},
{\it A continuity Theorem for Stinespring's dilation},
{\rm Journal of Functional Analysis 255 (2007)}.

\bibitem{neu-sh}
{\sc D.L. Neuhoff and P. C. Shields},
{\it Channel distances and representation},
{\rm Inform. Control 55 (1982), 238-264.}

\bibitem{Pa} 
{\sc V.\,I.\,Paulsen}, 
{\it Completely bounded maps and operator algebras}, 
{\rm Cambridge University Press, 2002}.

\bibitem{ryan}
{\sc R. A. Ryan},
{\it Introduction to tensor products of Banach spaces},
{\rm Springer Monographs in Mathematics, 2022. }

\bibitem{sakai}
{\sc S. Sakai},
{\it C*-algebras and W*-algebras},
{\rm Ergebnisse der Mathematik und ihrer Grenzgebiete, Nad 60. Springer-Verlag, New York-Heidelberg, 1971, 253.}

\bibitem{takesaki}
{\sc  M.Takesaki}
{\it Theory of operator algebras I},
{\rm Springer, 1979}.

\end{thebibliography}
\end{document}